\documentclass[11 pt]{amsart}
\usepackage{amssymb}
\usepackage{amsmath}
\usepackage{amsfonts}
\usepackage{graphicx}
\usepackage{amsthm}
\usepackage{enumerate}
\usepackage[mathscr]{eucal}
\usepackage{verbatim}

 \theoremstyle{plain}
\newtheorem{theorem}{Theorem}[section]
\newtheorem{lemma}[theorem]{Lemma}

\numberwithin{equation}{section}
\theoremstyle{plain}

\theoremstyle{remark}

\def\bbR{{\mathbb {R}}}

\begin{document}

\date{September, 2012}

\title
{Unit distance problems}

\author[]
{Daniel Oberlin and Richard Oberlin}

\address
{Daniel  Oberlin \\
Department of Mathematics \\ Florida State University \\
 Tallahassee, FL 32306}
\email{oberlin@math.fsu.edu}

\address
{Richard Oberlin \\
Department of Mathematics \\ Louisiana State University \\
Baton Rouge, LA 70803}
\email{oberlin@math.lsu.edu}

\subjclass{11B30, 42B10, 28E99}
\keywords{unit distance problem, dimension}

\thanks{D.O. was supported in part by NSF Grant DMS-1160680
and R.O. was supported in part by NSF Grant DMS-1068523.}

\begin{abstract}
We study some discrete and continuous variants of the following problem of Erd\H os:
given a finite subset $P$ of
$\bbR^2$ or $\bbR^3$, what is the maximum number of pairs $(p_1 ,p_2 )$ with 
$p_1 ,p_2 \in P$ and $|p_1 -p_2 |=1$?
\end{abstract}

\maketitle

\section{Introduction and Statement of Results}

In 1946 Paul Erd\H os \cite{E1} posed the following question:
given a finite subset $P$ of
 $\bbR^2$ or $\bbR^3$, what is the maximum number of pairs $(p_1 ,p_2 )$ with $p_1 ,p_2 \in P$ and $|p_1 -p_2 |=1$?
The Erd\H os unit distance conjecture in  $\bbR^2$ is the estimate
\begin{equation}\label{UDC}
\big|\{ (p_1 ,p_2 )\in P^2 :|p_2 -p_1 |=1\}\big| \le C\, |P|\sqrt{\log (|P|)}.
\end{equation}
(We will use $|\cdot |$ for the cardinality of a finite set as well as Lebesgue measure on $\bbR^d$.)
In two dimensions the best currently-known partial result, due to Spencer, Szemer\'edi, and Trotter \cite{SST}, is 
\begin{equation*}
\big|\{ (p_1 ,p_2 )\in P^2 :|p_2 -p_1 |=1\}\big| \le C\, |P|^{4/3},
\end{equation*}
while the current best estimate for the analogous problem in $\bbR ^3$ has the exponent $3/2+\epsilon $ 
(for any $\epsilon >0$ and $C$ depending on $\epsilon$) in
place of $4/3$ - see Clarkson {\it et al.}
\cite{CEGSW}. In four or more dimensions it follows from an example we learned in \cite{IJL} that one cannot significantly improve the trivial $|P|^2$ bound: let $\tilde P$ be any set of $N$ points ${\tilde x}_n$ in $\bbR^2$ satisfying $|{\tilde x}_n |=2^{-1/2}$. Let $P$ be the subset of $\bbR^4$ given by 
\begin{equation*}
P\dot= \{({\tilde x}_n ;0,0),(0,0;{\tilde x}_m ):{\tilde x}_n ,{\tilde x}_m \in \tilde P \}.
\end{equation*}
Then the left hand side of \eqref{UDC} is at least $N^2$ while $|P|^2 =4N^2$. 
Our first result shows that if we 
ban a salient feature of this example - many points in low-dimensional subspaces - then a nontrivial estimate is still possible:
\begin{theorem}\label{discretetheorem}
Fix $d\geq 2$. There is a positive constant $C_d$ such that if $P\subset \bbR^d$ and if every 
$d$-element subset of $P$ is affinely independent, then 
\begin{equation}\label{result}
\big|\{ (p_1 ,p_2 )\in P^2 :|p_2 -p_1 |=1\}\big| \leq C_d \, |P|^{(2d-1)/d}.
\end{equation}
\end{theorem}
\noindent (The proofs of the results described in this section can be found in \S \ref{proofs}.)

Another famous problem of Erd\H os is his distinct distance conjecture, the estimate
\begin{equation}\label{DDC}
\big|\{ |p_1 -p_2 |:(p_1 ,p_2 )\in P^2 \}\big| \ge c\, \frac{|P|}{\sqrt{\log (|P|)}}.
\end{equation}
An easy pigeon-hole argument shows that \eqref{UDC} implies \eqref{DDC}. But while the conjecture \eqref{UDC}  
is still far from resolved, Guth and Katz \cite{GK} have recently come very close to \eqref{DDC} by showing that  
\begin{equation*}
\big|\{ |p_1 -p_2 |:(p_1 ,p_2 )\in P^2 \}\big| \ge c\, \frac{|P|}{{\log (|P|)}}.
\end{equation*}

This distinct distance problem has a continuous analog known as the Falconer distance set problem (\cite{F}): 
if $K$ is a compact subset 
of $\bbR^d$ and if we define the distance set $\Delta (K)$ by
\begin{equation*}
\Delta (K)=\{ |k_1 -k_2 |: (k_1 ,k_2 )\in K^2 \},
\end{equation*}
then what can we say about lower bounds for $\dim \big(\Delta (K)\big)$ in terms of $\dim (K)$? 
For example, Wolff proves in \cite{W} that 
if $K\subset\bbR^2$ and $\dim (K)>4/3$ then $\Delta (K)$ has positive Lebesgue measure and so dimension one, while 
Erdo\u{g}an \cite{B1} contains analogous results in $\bbR^d$. 

The primary purpose of this paper is to study the following continuous analog of the {\bf unit} distance problem: if 
\begin{equation*}\label{Ddef}
D=D(K)=\{(k_1 ,k_2 )\in K^2 :|k_2 -k_1 |=1\}, \, K\subset \bbR^d ,
\end{equation*}
find
\begin{equation}\label{g_d}
g_d (\alpha )\dot=\sup \{\dim (D): \text{$K$ is a compact subset of $\bbR^d$ with $\dim (K)=\alpha$}\}.
\end{equation}
When $d=1$ this is trivial: the projection $(k_1 ,k_2)\mapsto k_1$ is at most two-to-one on $D$ and so 
it follows that $\dim (D)\le\alpha$. 
If $\tilde K \subset\bbR$, $\dim (\tilde K )=\alpha$, and if $K=\tilde K \cup (\tilde K +1)$, then
$\dim (D)=\alpha =\dim (K)$. Thus $g_1 (\alpha )=\alpha$.

Here is a trivial bound in higher dimensions: the map 
$$
(k_1 ,k_2 )\mapsto (k_1 ,k_2 -k_1 )
$$
shows that $D$ and 
\begin{equation}\label{Gdef}
G\doteq\{(k,y): k\in K,\ y\in S^{d-1},\ k+y\in K\}
\end{equation}
have the same dimension. This gives the bound
\begin{equation}\label{trivbd}
\dim (D)\le \alpha +d-1 .
\end{equation}
More interestingly, 
$D$ is the intersection of  $K\times K$ with the variety 
\begin{equation*}
\{(x_1 ,x_2 )\in \bbR^d \times \bbR^d :|x_2 -x_1 |=1\}.
\end{equation*}
Thus one might conjecture that 
\begin{equation*}
\dim (D)\le 2\alpha -1
\end{equation*}
and so
\begin{equation}\label{largealpha.}
g_d (\alpha )\le 2\alpha -1.
\end{equation}
Of course this cannot always be correct since $g_d (\alpha )\ge \alpha$ 
if $0\le\alpha\le 1$ (because $g_d (\alpha )\geq g_1 (\alpha )$ 
since $\bbR^d$ contains a copy of $\bbR$). 
But here is an example related to \eqref{largealpha.}: 
suppose $C\subset B(0,1/2)\subset \bbR^{d-1}$ has $\dim (C)=\gamma$ and put 
$K=C\times [0,2]\subset\bbR^d$. Then $\alpha =\dim (K)=1+\gamma$. Also 
$$
D=\{(c_1 ,t_1 ;c_2 ,t_2):c_1 ,c_2 \in C,\ t_1 ,t_2 \in [0,2],\, |t_1 -t_2 |=\sqrt{1-|c_1 -c_2 |^2}\}.
$$
Since for each fixed $(c_1 ,t_1 ;c_2)$ with $c_1 ,c_2 \in C,\ 0\le t_1 \le 1$ there is a $t_2 \in [0,2]$ which works in 
$|t_1 -t_2 |=\sqrt{1-|c_1 -c_2 |^2}$, it follows that
$$
\dim (D)=\dim (C\times C )+1\ge 2\gamma+1 =2\alpha-1. 
$$
Thus when $\alpha\ge 1$ it is at least not possible to do better than \eqref{largealpha.}. This example has another implication too: there are sets $C\subset \bbR$ with $\dim (C)=0$ and $\dim (C\times C)=1$. (That is a manifestation of the fact that Hausdorff dimension does not always behave well when forming Cartesian products.) It follows that there are sets $K\subset \bbR^2$ with $\dim (K)=1$ and $\dim (D)=2$, discouraging news when looking for something better than the trivial estimate \eqref{trivbd}.
To rule out this sort of degeneracy we will assume that our $\alpha$-dimensional sets $K$ have a certain regularity - defining 
$K_\delta =K+B(0,\delta )$, we will assume for the remainder of the paper that $K_\delta$ is a $\delta$-discrete $\alpha$-set in the sense of Katz and Tao \cite{KT2}. This means that 
\begin{equation}\label{alphaset}
|K_\delta \cap B(x,r)|\le C(K) \, (r/\delta )^\alpha \delta^d 
\end{equation}
for any $x\in\bbR^d $ and $r\ge \delta $.  In particular, we will now assume that the $\alpha$-dimensional sets figuring in \eqref{g_d} all satisfy \eqref{alphaset}. 
With the assumption \eqref{alphaset} in place we will obtain some nontrivial estimates on the upper Minkowski dimension $\dim_M (D)$ of $D$. But first we record another trivial estimate.
Since $|K_\delta |\lesssim \delta^{d-\alpha}$ by \eqref{alphaset}, it follows from $D\subset K\times K$
that $|D_\delta |\lesssim \delta^{2d-2\alpha}$. Thus $\dim_M (D)\le 2\alpha$ and so  
\begin{equation}\label{trivest.}
g_d (\alpha )\le 2\alpha .
\end{equation}

Our first nontrivial bound for $g_d$ concerns large values of $\alpha$:
\begin{theorem}\label{largealpha}
If $(d+1)/2 \le \alpha \le d$, then $g_d (\alpha )= 2\alpha -1$; if $\alpha \le (d+1)/2$, then 
$g_d (\alpha )\le \alpha +(d-1)/2$.   
\end{theorem}
\noindent The proof uses the Fourier transform. 
The second statement of Theorem \ref{largealpha} is only interesting when $\alpha +(d-1)/2$ is less than the $2\alpha$ in \eqref{trivest.} and so only when $\alpha >(d-1)/2$.
On the other hand, the first statement of Theorem \ref{largealpha} shows that
the conjecture \eqref{largealpha.} is correct for $\alpha\ge (d+1)/2$. In particular, and in contrast to the discrete unit distance problem, when $\alpha$ is sufficiently large there are positive results 
available in $\bbR^d$ even when $d\ge 4$. But the same example
which rules out positive results on the discrete unit distance problem for $d\ge 4$ can be easily modified to show that there are no 
nontrivial results on the continuous problem when $d\ge 4$ and $\alpha$ is small. In particular we have the following statement.
\begin{equation}\label{example1}
\text{If $d\ge 4$ and $\alpha \le \lfloor d/2 \rfloor -1$, then $g_d (\alpha )=2\alpha$.}
\end{equation}
(To see why \eqref{example1} is true, first note that the inequality 
$g_{d+1}(\alpha )\ge g_d (\alpha )$ shows that it is enough to consider only the case when $d$ is even. In this case let $\tilde K$
be an appropriate $\alpha$-dimensional subset of $S^{ d/2  -1}\subset\bbR^{d/2}$ and define $K$ by
$$
K=2^{-1/2}\{(\tilde k_1 ,0),\, (0,\tilde k_2 )\in \bbR^{ d/2}\times \bbR^{ d/2}: \tilde k_1 ,\tilde k_2 \in \tilde K \}.)
$$
If $d\ge 4$ and $\alpha\in (\lfloor \frac{d}{2}\rfloor -1 ,\frac{d-1}{2})$ we do not know if the trivial estimate 
\eqref{trivest.} can be improved.

For $d=2$ or $d=3$ we have the following theorems, which contain nontrivial results for small $\alpha$. 
\begin{theorem}\label{d=2}
For $0<\alpha \le 1$ we have 
\begin{equation*}\label{2dest}
\frac{3\alpha }{2}\le g_2 (\alpha )\le \min \Big\{\frac{5\alpha}{3} ,\frac{\alpha (2+\alpha )}{1+\alpha}\Big\}. 
\end{equation*}
Additionally, for $1\le\alpha \le 3/2$ we have $g_2 (\alpha )=\alpha +1/2$ and for 
$3/2\le\alpha\le 2$ we have $g_2 (\alpha )=2\alpha -1$.
\end{theorem}
\noindent Except for the fact that 
$g_2 (\alpha )\ge \alpha +1/2$ when $1\le \alpha \le 3/2$, the second statement here is a 
consequence of Theorem \ref{largealpha}.
Parts of the proofs of Theorem \ref{d=2} and of Theorem \ref{d=3} below employ incidence geometry in the continuous setting - see \cite{S} for other examples.

\begin{theorem}\label{d=3}
We have 
$g_3 (\alpha )\le \frac{15\alpha}{8}$.
\end{theorem}
\noindent We note that, in addition to improving \eqref{trivest.} and 
improving \eqref{trivbd} for $\alpha <16/7$,  
the estimate in Theorem \ref{d=3} improves the second bound in Theorem \ref{largealpha} when $\alpha\le 8/7$.

\section{Proofs}\label{proofs}

{\it Proof of Theorem \ref{discretetheorem}:}
Modifying \eqref{Gdef} to fit the context of Theorem \ref{discretetheorem} gives
\begin{equation*}
G=\{(p,b): p\in P,\, b\in S^{d-1},\,  p+b\in P\}.
\end{equation*}
The correspondence $(p_1 ,p_2 )\longleftrightarrow (p_1 ,b)\dot=(p_1 ,p_2 -p_1 )$ shows that \eqref{result} is equivalent to 
\begin{equation}\label{result'}
\big|G\big| \leq C_d \, |P|^{(2d-1)/d}.
\end{equation}
Define
\begin{equation*}\label{Vdef}
V\dot= \{(p,b_1 ,\dots ,b_d ):(p,b_j )\in G,\, j=1,\dots ,d\}.
\end{equation*}
Then \eqref{result'} is a consequence of the two inequalities
\begin{equation}\label{result1}
\frac{|G|^d}{|P|^{d-1}}\le |V|
\end{equation}
and
\begin{equation}\label{result2}
|V|\leq C_d \, |P|^d .
\end{equation}
Inequality \eqref{result1} follows from a H\"older's inequality argument in the spirit of \cite{KT}:
\begin{equation*}\label{ineq1}
|G|=\sum_{p\in P,|b|=1}\chi_{G}(p,b)\leq\, 
\Big(\sum_{p\in P} \big(\sum_{|b|=1} \chi_G (p,b)\big)^d \Big)^{1/d} \,\big|P|^{(d-1)/d}.
\end{equation*}
To see \eqref{result2}, write $V$ as the disjoint union $V' \cup V''$ where $V'$ is the subset of $V$
consisting of all $(p,b_1 ,\dots ,b_d )$ for which $b_i =b_j$ for some $i\not= j$. Since $(p,b)\in G$ 
implies $b\in P-p$, it is clear that 
\begin{equation*}
|V' |\leq C_d \, |P|^d .
\end{equation*}
To obtain a similar estimate for $V''$, consider the mapping 
\begin{equation*}
\Phi:(p,b_1 ,\dots ,b_d )\mapsto (p+b_1 ,\dots ,p+b_d )
\end{equation*}
of $V''$ into $P^d$. It will be enough to show that $\Phi$ is at most two-to-one. 
Since $(p,b)\in G$ implies $b\in P-p$, it follows from 
 $(p,b_1 ,\dots ,b_d )\in V''$ that there are distinct $p_1 ,\dots ,p_d \in P$ such that
\begin{equation*}
(b_2 -b_1 ,\dots ,b_d -b_1 )=(p_2 -p_1 ,\dots ,p_d -p_1 )\dot= (a_2 ,\dots ,a_d ).
\end{equation*}
Our hypothesis concerning affine independence implies that the vectors $a_2 ,\dots ,a_d $
are linearly independent.
Next, suppose that 
\begin{equation*}
\Phi(p',b'_1 ,\dots ,b'_d )=\Phi(p,b_1 ,\dots ,b_d ).
\end{equation*}
Then
\begin{equation*}
b'_j -b'_1 =(p'+b'_j )-(p'+b'_1 )=(p +b_j )-(p +b_1 )=a_j
\end{equation*}
for $j=2,\dots ,d$. The desired multiplicity estimate for $\Phi$ now follows from Lemma \ref{lemma1} 
below (an analog  of the fact that there are at most two chords of a circle which are congruent under translation).

\begin{lemma}\label{lemma1}
Suppose that $a_2 ,\dots ,a_d \in \bbR ^d$ are linearly independent. Then there are at most two $d$-tuples
$(b_1 ,\dots ,b_d )$ with $b_j \in\bbR^d$ such that 
\begin{equation}\label{bcond}
|b_1 |=\cdots =|b_d|=1,\text{and }b_j -b_1 =a_j  ,\  j=2,\dots ,d.
\end{equation}
\end{lemma}
\begin{proof}
Let $H$ be the hyperplane in $\bbR^d$ spanned by $a_2 ,\dots ,a_d$ and fix a nonzero vector $v$ with
$v\perp H$.
Our first goal is to prove the following statement:
\begin{multline}\label{2t's}
\text{there is $\{t_1 ,t_2\}\subset\bbR$ depending only on $\{ a_2 ,\dots ,a_d \}$ and $v$} \\  
\text{such that if \eqref{bcond}
holds, then $\{ b_1 ,\dots ,b_d \}\subset (tv+H) \cap S^{d-1}$} \\
\text{for some $t\in\{t_1 ,t_2 \}$.
\ \ \ \ \ \ \ \ \ \ \ \ \ \ \ \ \ \ \ \ \ \ \ \ \ \ \ \ \ \ \ \ \ \ \ \ \ \ \ \ \ \ \ \ \ \ \ \ \ \ \ \ \ \ \ \ \ \ } 
\end{multline}
To see \eqref{2t's} we begin by noting that 
if $w\in\bbR^d$ then the intersection 
$$
(tw+H) \cap S^{d-1}
$$ 
is either a $(d-2)$-sphere or empty.
In particular, if $r_0$ is the radius of the $(d-2)$-sphere determined by $\{0,a_2 ,\dots ,a_d \}$
(the linear independence of $a_2, \ldots, a_d$ guarantees that there is 
exactly one $(d-2)$-sphere containing these points)
then there is $\{t_1 ,t_2 \}\subset \bbR$, depending only on $a_2 ,\dots ,a_d$ and $v$,
such that if $(tv+H) \cap S^{d-1}$ is a $(d-2)$-sphere of radius $r_0$
exactly when $t\in\{t_1 ,t_2 \}$.
Suppose that \eqref{bcond} holds.
Since $(b_1 +H) \cap S^{d-1}$ contains $b_1$,  $(b_1 +H) \cap S^{d-1}$ is a $(d-2)$-sphere. Since  
\begin{equation*}
b_1 +\{ 0,a_2 ,\dots ,a_d \}\subset (b_1 +H) \cap S^{d-1}, 
\end{equation*}
it follows that 
$(b_1 +H) \cap S^{d-1}$ is a $(d-2)$-sphere of radius $r_0$. Then (by \eqref{bcond})
\begin{equation*}
\{b_1 ,\dots ,b_d \}=b_1 +\{ 0,a_2 ,\dots ,a_d \}\subset (b_1 +H) \cap S^{d-1}=(tv+H) \cap S^{d-1}
\end{equation*}
for some  $t\in\{t_1 ,t_2 \}$. This establishes \eqref{2t's}.

Given \eqref{2t's}, the proof of the lemma will be complete if we show that for fixed $t\in\bbR$ there is at most 
one $d$-tuple $(b_1 ,\dots ,b_d )$ such that both
\eqref{bcond} and 
\begin{equation}\label{bcond1}
\{b_1 ,\dots ,b_d \}\subset (tv+H) \cap S^{d-1}
\end{equation}
hold. So suppose that \eqref{bcond} and \eqref{bcond1} hold for $(b_1 ,\dots ,b_d )$ 
and also for $(b'_1 ,\dots ,b'_d )$. Let $r_0$ and $c$
be the radius and center of the $(d-2)$-sphere $(tv+H) \cap S^{d-1}$.
Then the points $\{0,a_2,\ldots,a_d\}$ are contained in the $(d-2)$-spheres 
in $H\cap S^{d-1}$
of radius $r_0$ centered at $c-b_1$ and $c-b'_1$. 
Again appealing to the fact that $\{0,a_2,\ldots, a_d\}$ determines a unique $(d-2)$-sphere we see that $b_1 = b'_1$ and hence $(b_1,\ldots,b_d) = (b'_1, \ldots, b'_d).$
\end{proof}

{\it Proof of Theorem \ref{largealpha}:}
Recalling the definition \eqref{g_d} of $g_d$, we will bound $g_d$ by estimating 
$\dim_M (D)$. Since $\dim_M (D)\le \gamma$ will follow from 
$|D_\delta |=|D+B(0,\delta )|\lesssim \delta^{2d-\gamma-\epsilon}$
for all $\epsilon >0$ and since $D_\delta \subset D^\delta$
where $D^\delta$ is defined by  
\begin{equation*}
D^\delta \doteq \{(k_1 ,k_2 )\in K_\delta \times K_\delta : 1-2\delta\le |k_2 -k_1 |\le 1+2\delta \},
\end{equation*}
we will be interested in estimating $|D^\delta |$.
Without loss of generality assume that $K=-K$ and write
\begin{equation}\label{Dest}
|D^\delta |=\int_{K_\delta}\int_{K_\delta}1_{A(0,\delta )}(x_2 -x_1 )\, dx_1 \, dx_2 =
\langle 1_{K_\delta}\ast 1_{K_\delta},1_{A(0,\delta )}\rangle
\end{equation}
where, for $c\in\bbR^d$,  $A(c,\delta ) =\{x\in\bbR^d :1-2\delta \le |x-c|\le 1+2\delta \}$. Let $\rho$ be a symmetric Schwartz function  with 
\begin{equation}\label{rhoprops}
1_{B(0,C)}\lesssim |\hat\rho |\lesssim 1_{B(0,2C)},\ 
1_{B(0,C' )}(x)\le \rho (x) 
\lesssim \sum_{j=1}^\infty 2^{-jd}\,1 _{B(0,2^j )}(x) .
\end{equation}
Write $\sigma$ for Lebesgue measure on $S^{d-1}$. 
If $\rho_r (x)=r^{-d}\, \rho (x/r)$, and if $C'$ is chosen appropriately, then 
\begin{multline}\label{Dest2}
|D^\delta |\lesssim \delta\,\langle 1_{K_\delta}\ast1_{K_\delta},\rho_\delta  \ast 
\rho_\delta \ast \sigma\rangle =\delta\,
\langle (1_{K_\delta}\ast\rho_\delta)\ast (1_{K_\delta}\ast\rho_\delta), \sigma\rangle  \lesssim \\
\delta\, \int_{B(0,2C/\delta )}
\big| \widehat {1_{K_\delta}\ast \rho_\delta}(\xi )\big|^2 \, \frac{d\xi}{1+|\xi |^{(d-1)/2}}.
\end{multline}
We will control the last integral by estimating $\|1_{K_\delta}\ast \rho_\delta \|_2$ 
and we begin by estimating $\|1_{K_\delta}\ast \chi_{B(0,r)}\|_2$ for $r\ge\delta$.
Using \eqref{alphaset} we have 
$$
\|1_{K_\delta}\ast 1_{B(0,r)}\|_1 
\lesssim r^d \, \delta^{d-\alpha},\ \|1_{K_\delta}\ast 1_{B(0,r)}\|_\infty
\lesssim r^\alpha \delta^{d-\alpha},
$$
and so 
\begin{equation}\label{est} 
\|1_{K_\delta}\ast 1_{B(0,r)}\|_2 \lesssim r^{(d+\alpha )/2}\,\delta^{d-\alpha},\ r\ge \delta .
\end{equation}
Then \eqref{rhoprops} and \eqref{est} show that
for $r\ge \delta$ we have 
\begin{equation}\label{est2}
\|1_{K_\delta}\ast \rho_r\|_2 \lesssim r^{(\alpha -d)/2}\delta^{d-\alpha}.
\end{equation}
Since $\hat{\rho_r}$ is supported on $B(0,C/r)$, \eqref{est2} implies
\begin{equation*}
\int\limits_{\small\frac{C}{2r}\le |\xi |\le \small\frac{C}{r}}
\big| \widehat {1_{K_\delta}\ast \rho_\delta}(\xi )\big|^2 \, d\xi \lesssim
\int\limits_{\small\frac{C}{2r}\le |\xi |\le \small\frac{C}{r}}
\big| \widehat {1_{K_\delta}\ast \rho_r}(\xi )\big|^2 \, d\xi\lesssim r^{\alpha -d}\, \delta^{2(d-\alpha )},\ r\ge \delta .
\end{equation*}
Thus 
\begin{multline}\label{est3}
\int\big| \widehat {1_{K_\delta}\ast \rho_\delta}(\xi )\big|^2 \, \frac{d\xi}{|\xi |^{d-\alpha}}=
\int\limits_{|\xi |\le \small\frac{2C}{\delta}}
\big| \widehat {1_{K_\delta}\ast \rho_\delta}(\xi )\big|^2 \, \frac{d\xi}{|\xi |^{d-\alpha}}= \\
\int\limits_{\{|\xi |\le C\}}
\big| \widehat {1_{K_\delta}\ast \rho_\delta}(\xi )\big|^2 \, \frac{d\xi}{|\xi |^{d-\alpha}}+
\sum\limits_{\small\frac{1}{2}\le 2^j \le\small\frac{1}{\delta}}
\int\limits_{\small\frac{C}{2^{j+1} \delta}\le |\xi |\le \small\frac{C}{2^{j}\delta}}
\big| \widehat {1_{K_\delta}\ast \rho_\delta}(\xi )\big|^2 \, \frac{d\xi}{|\xi |^{d-\alpha}}\lesssim \\
\delta^{2(d-\alpha )}+\sum\limits_{\small\frac{1}{2}\le 2^j \le\small\frac{1}{\delta}}
\Big(\frac{C}{2^j \delta}\Big)^{\alpha -d}\big(2^j \delta \big)^{\alpha -d}\,\delta^{2(d-\alpha )}\lesssim 
\log (\small\frac{1}{\delta}\big)\,\delta^{2(d-\alpha )}.
\end{multline}
(So normalized Lebesgue measure on $K_\delta$ behaves like an $\alpha$-dimensional measure from the Fourier transform point of view.) We will use \eqref{est3} to estimate $|D^\delta |$ via \eqref{Dest2} 
and thus to obtain the upper bounds on $g_d$ in Theorem \ref{largealpha}.  
If $\alpha\ge (d+1)/2$, so that $(d-1)/2 \ge d-\alpha$, then the integral in \eqref{Dest2} is dominated by the integral estimated in \eqref{est3}. That leads to 
$|D^\delta |\lesssim \log (\small\frac{1}{\delta}\big)\,\delta^{2(d-\alpha )+1}$ and so, by the remarks at the beginning of this proof, to $g_d (\alpha )\le 2\alpha -1$. With the example described 
after \eqref{largealpha.}, this gives 
$g_d (\alpha )= 2\alpha -1$. If $\alpha\le(d+1)/2$, then when $|\xi |\le C/\delta$ we have
$$
\frac{1}{|\xi |^{(d-1)/2}}
\lesssim\frac{ \delta^{\alpha -(d+1)/2}}{|\xi |^{d-\alpha}}
$$
which leads as above to $g_d (\alpha )\le \alpha +(d-1)/2$.

{\it Proof of Theorem \ref{d=2}:}
We begin by claiming that it is enough to prove the upper bounds for $g_2 (\alpha )$ under the additional assumption that 
\begin{equation*}\label{diam}
\text{diam} (K)\le 2-\eta
\end{equation*}
for some fixed $\eta >0$. (The purpose of this restriction is to avoid the possibility of external
tangencies of certain annuli and thus to allow the use of estimates like \eqref{intersectionest} below.)
To see that this reduction is legitimate, let $\{C_1 ,\dots ,C_7 \}$ be a partition of the unit circle into arcs each having length less than $.9$ and let 
$$
G_i =B(0,1/100) \cup \big(B(0,1/100)+C_i\big) .
$$
Then
$$
D\subset\cup_{k\in K}\cup_{1\le i\le 7}\{(k_1 ,k_2 )\in (k+G_i )^2 \}.
$$
As $D$ is compact, it is contained in some finite union of sets 
$$
\{(k_1 ,k_2 )\in (k+G_i )^2 \}.
$$
Since $\text{diam} \big(k+G_i \big)\le 2-\eta$ for some fixed $\eta >0$, our claim is established.

By renaming $\eta$ and assuming that $\delta>0$ is small enough, we can (and do) assume 
for the remainder of this proof that 
\begin{equation}\label{diam'}
\text{diam} (K_\delta )\le 2-\eta .
\end{equation}

We now turn to the proof of the upper bound 
\begin{equation}\label{UB1}
g_2 (\alpha )\le \frac{\alpha (2+\alpha )}{1+\alpha}
\end{equation}
Under the assumption that $K$ satisfies \eqref{alphaset} for $d=2$, it is enough to establish the estimate
\begin{equation*}
|D^\delta |\lesssim \log (1/\delta )\,\delta ^{4-\alpha (2+\alpha )/(1+\alpha)}. 
\end{equation*}
(Throughout this argument the constants implied by the symbol $\lesssim$ depend only on $K$).
With 
\begin{equation*}\label{Gdelta}
G^\delta = \{(k,y): k\in K_\delta , \, 1-2\delta \le |y|\le 1+2\delta, \, k+y\in K_\delta \}, 
\end{equation*}
the correspondence $(k_1 ,k_2 )\longleftrightarrow (k_1 ,y )\dot= (k_1 ,k_2 -k_1 )$ shows that 
$|D^\delta |=|G^\delta |$. Thus it suffices to show that
\begin{equation}\label{UB1.1}
|G^\delta |\lesssim \log (1/\delta )\,\delta ^{4-\alpha (2+\alpha )/(1+\alpha)}. 
\end{equation}
Recall that $A(0,\delta ) =\{x\in\bbR^d :1-2\delta \le |x|\le 1+2\delta \}$. For $k\in K_\delta$ we will write 
$(G^\delta )_k$ for the $k$-section of $G^\delta$ given by $\{y\in A(0,\delta ):k+y\in K_\delta \}$.
Since $K_\delta$ is a $\delta$-discrete $\alpha$-set we can assume that the two-dimensional 
Lebesgue measure of $(G^\delta )_k$ satisfies $\delta^2 \lesssim |(G^\delta )_k |\lesssim \delta^{2-\alpha}$. 
Find $M\lesssim \log (1/\delta )$ positive numbers $\{\lambda_m \}_{m=1}^M$ such that $\lambda_{m+1}=2\,\lambda_m$
and such that for each $k\in K_\delta$ we have $\lambda _m \le |(G^\delta )_k |\le \lambda_{m+1}$ for some $m$.
Then define 
$$
K^m = \{k\in K_\delta : \lambda _m \le |(G^\delta )_k |\le \lambda_{m+1}\}.
$$
The estimate \eqref{UB1.1} (of the four-dimensional Lebesgue measure of $G^\delta$) will follow from the following estimate of the two-dimensional Lebesgue measure of $K^m$:
\begin{equation}\label{UB1.2}
\lambda_m \,|K^m |\lesssim \delta^{4-\alpha (2+\alpha )/(1+\alpha)}.
\end{equation}

So fix $m$. Choose $N=N(m)$ disjoint balls $B(c_n ,\delta )$ with $c_n \in K_m$ for which 
\begin{equation}\label{UB1.3}
\lambda\doteq \lambda_m\le 
|A(c_n ,\delta ) \cap K_\delta |
\end{equation}
and such that 
\begin{equation}\label{UB1.35}
|K^m |\lesssim N\, \delta^2 .
\end{equation}
Our goal is the estimate 
\begin{equation}\label{UB1.4}
\lambda\,N\,\delta^2 \lesssim \frac{\delta^{4}}{\lambda^{\alpha}} 
\end{equation}
which when interpolated with the trivial estimate
\begin{equation}\label{trivest}
\lambda\,N\,\delta^2 \lesssim \lambda\, \delta^{2-\alpha}
\end{equation}
gives \eqref{UB1.2} via \eqref{UB1.35}.

To prove \eqref{UB1.4} we begin by fixing $r=C\delta^2 /\lambda$. 
Since $|K_\delta \cap B(x,r)|\lesssim (r/\delta )^\alpha \delta^{2}$, any
$B(x,r)$ contains $\lesssim (r/\delta )^\alpha$ of the $B(c_n ,\delta )$'s. Thus there is an $r$-separated  subcollection 
$\{\tilde c_n \}$ containing $\tilde N$ of the  of the $c_n$'s, where 

\begin{equation}\label{Ntilde}
N\delta^\alpha /r^\alpha \lesssim \tilde N  .
\end{equation}
The bound \eqref{UB1.4} will follow from a certain estimate from below of the two-dimensional Lebesgue measure
\begin{equation*}
|\cup_n \big(A(\tilde c_n ,\delta ) \cap K_\delta \big)|. 
\end{equation*}
Part of the strategy here is the general estimate 
\begin{equation}\label{unionest}
|\cup_n E_n |\ge \sum_n |E_n | -\sum_{n_1 <n_2 }|E_{n_1}\cap E_{n_2}|.
\end{equation}
We will take $E_n =A(\tilde c_n ,\delta ) \cap K_\delta$ and use the estimate 
\begin{equation}\label{intersectionest}
|A(\tilde c_{n_1},\delta ) \, \cap \ A(\tilde c_{n_2},\delta )|
\lesssim\frac{\delta^2}{\delta+|\tilde c_{n_1}-\tilde c_{n_2}|}
\end{equation}
(in which the implied constant depends on $\eta$ in \eqref{diam'}\,)
to bound $|E_{n_1}\cap E_{n_2}|$. For this reason we are interested in controlling the quantity 
\begin{equation*}\label{sum}
\sum_{n\not= n_0}\frac{\delta^2}{ |\tilde c_{n}-\tilde c_{n_0}|}.
\end{equation*}
We are assuming that the sets $K_{\delta '}$ are unifomly $\delta '$-discrete - that they satisfy \eqref{alphaset} uniformly in $\delta '$ - and so, in particular, $K_r$ is $r$-discrete. Thus for each $\tilde c_{n_0}$ there are at most 
$C_2 \,  2^{k\alpha}$ of the $r$-separated $\tilde c_n$' s within distance $2^k r$ of $\tilde c_{n_0}$.
Therefore, since $\alpha <1$,
\begin{equation*}
\sum_{n\not= n_0}\frac{\delta^2}{ |\tilde c_{n}-\tilde c_{n_0}|}\lesssim \delta^2  \sum _{k=1}^\infty \frac{2^{k\alpha}}{2^k r}\lesssim 
\frac{\delta^2 }{r}
\end{equation*}
and so 
\begin{equation}\label{sumest}
\sum_{n\not= n_0}\frac{\delta^2}{ |\tilde c_{n}-\tilde c_{n_0}|}\le c\,\lambda
\end{equation}
by our choice of $r$. Thus \eqref{intersectionest} and \eqref{sumest} imply
\begin{equation*}
\sum_{n_1 <n_2 }|A(\tilde c_{n_1},\delta ) \, \cap \ A(\tilde c_{n_2},\delta )|\leq \\
 C'\tilde{N} c\, \lambda  = \tilde{N}c' \lambda.
\end{equation*}
On the other hand, because of \eqref{UB1.3} and \eqref{Ntilde} we have  
\begin{equation*}
\sum_n |A(\tilde c_n ,\delta )\cap K_\delta |\ge \tilde{N} \lambda 
\end{equation*}
and so, by \eqref{unionest}, 
\begin{equation*}
|\cup_n \big(A(\tilde c_n ,\delta )\cap K_\delta \big)|\ge (1 - c') \tilde{N} \lambda \gtrsim
(1 - c') \Big(\frac{N\delta^\alpha}{r^\alpha}\Big) \lambda .
\end{equation*}
If $C$ (figuring in the choice of $r$) is large enough, then $1 - c' >0$ and so this 
last estimate and the fact that 
$|K_\delta |\lesssim \delta^{2-\alpha}$, together with our choice of $r$, yield \eqref{UB1.4}. 
This completes the proof of \eqref{UB1}.

Next we give the proof of the upper bound 
\begin{equation}\label{UB2}
g_2 (\alpha )\le \frac{5\alpha}{3}.
\end{equation}
Part of the argument is analogous to the proof of Theorem \ref{discretetheorem}.
Let $K_m$, $\lambda =\lambda_m$, and the $B(c_n ,\delta ),\,  1\le n\le N,$ be as in the proof of \eqref{UB1}.
Instead of \eqref{UB1.4} we will now prove
\begin{equation}\label{UB2.1}
\lambda\,N\,\delta^2 \lesssim  \frac{\delta^{2+3(2-\alpha )}}{\lambda ^2}.
\end{equation}
As above, interpolation with \eqref{trivest} will then lead to 
\begin{equation*}
|D^\delta |\lesssim \log (1/\delta )\,\delta ^{4-5\alpha /3} 
\end{equation*}
and so to \eqref{UB2}.  

Choose a maximal $\delta$-separated subset $J$ of $K_\delta$. For each $c_n$ let 
\begin{equation*}
S_{c_n}=\{a\in J :1-3\delta \le |a-c_n |\le 1+3\delta \},
\end{equation*}
so that $S_{c_n}$ is like a discretized $c_n$-section of $D^\delta$. 
Define 
\begin {equation*}
V=\big\{(c_n ,a_1 ,a_2 ): 1\le n\le N,\, a_1 ,a_2 \in S_{c_n},\, |a_1-a_2 |\ge 
c\, \Big(\small\frac{\lambda}{\delta^{2-\alpha}}\Big)^{1/\alpha }\big\},
\end{equation*}
where $c$ is a small positive constant. We will prove \eqref{UB2.1} by comparing upper and 
lower estimates for $|V|$. 

Since 
\begin{equation*}
|\{k\in K_\delta : 1-2\delta \le |k-c_n | \le 1+2\delta\}|\ge\lambda
\end{equation*}
by the choice of $c_n$ it follows that $|S_{c_n}|\gtrsim \lambda /\delta^2$. Since \eqref{alphaset} implies that 
\begin{equation*}
|K_\delta \cap B\big(a, c(\lambda / \delta^{2-\alpha })^{1/\alpha}\big)| \lesssim c^\alpha \, \lambda
\end{equation*}
for any $a$, it follows that
\begin{equation}\label{Tlower}
|V|\gtrsim N\,\Big(\frac{\lambda}{\delta^2}\Big)^2 
\end{equation}
if $c$ is small enough.

To obtain an upper bound for $|V|$ we begin by noting that if 
$$
(c_{n_0} ,a_1 ,a_2 ) \in V
$$
then $c_{n_0}$ is in 
\begin{equation}\label{inter}
A(a_1 ,3\delta )\cap A(a_2 ,3\delta ).
\end{equation}
%
%contains a $\delta$-neighborhood of the segment joining 
Because $|a_1 -a_2 |\le 2-\eta <2$ it follows that if $|a_1 -a_2 |\gtrsim \delta$ then \eqref{inter} is a union of two connected components, one on either side of the line through $a_1$ and $a_2$ and each having diameter bounded above by 
\begin{equation}\label{multiplicity}
C \frac{\delta}{|a_1 -a_2 |}\lesssim \Big(\frac{\delta^2}{\lambda}\Big)^{1/\alpha},
\end{equation}
where the inequality comes from the definition of $V$. The hypothesis \eqref{alphaset} then implies that each connected component of \eqref{inter} contains $\lesssim \delta^{2-\alpha}/\lambda$ points from $\{c_n\}$. Thus the projection 
$$
(c_n ,a_1 ,a_2 )\mapsto (a_1 ,a_2 )
$$
of $V$ into $J\times J$ has multiplicity at most $C\, \delta^{2-\alpha}/\lambda$. Therefore
\begin{equation}\label{Tupper}
|V|\lesssim |J|^2 \, \frac{\delta^{2-\alpha}}{\lambda}\lesssim \delta^{-2\alpha}\, \frac{\delta^{2-\alpha}}{\lambda}.
\end{equation}
Comparison of \eqref{Tlower} and \eqref{Tupper} yields \eqref{UB2.1}. 
This completes the proof of \eqref{UB2}.

To complete the proof of Theorem \ref{d=2} we need to establish the two lower bounds on $g_2 (\alpha )$
\begin{equation}\label{LB1}
\text{$g_2 (\alpha )\ge 3\alpha /2$ if $0<\alpha \le 1$ }
\end{equation}
and
\begin{equation}\label{LB2}
\text{$g_2 (\alpha )\ge \alpha +1/2$ if $1<\alpha \le 3/2$. }
\end{equation}
These will be consequences of the following lemma.
\begin{lemma}
Suppose $0<\beta , \gamma <1$ are rational and let $\alpha '=\beta +\gamma$. There is a compact set 
$K\subset \bbR^2$ which satisfies \eqref{alphaset} with $\alpha '$ instead of $\alpha$ and for which we have
 $|D^\delta |\gtrsim \delta^{4-(\beta +3\gamma /2)}$ for some sequence of $\delta$'s tending to $0$. 
\end{lemma}
\noindent To deduce \eqref{LB1}, approximate $\alpha$ by $\alpha '$ with $\beta$ very close to $0$;
to deduce  \eqref{LB2}, approximate $\alpha$ by $\alpha '$ with $\gamma$ very close to $1$.

\begin{proof} We will require compact subsets $A,B\subset [0,1]$ which satisfy \eqref{alphaset} with $\alpha$
replaced by $\beta$ in the case of $A$ and by $\gamma$ in the case of $B$. We will also need $A$ and $B$ to satisfy 
the two lower bounds
\begin{equation}\label{L0}
\int_{A_{\delta_n}}\int_{A_{\delta_n}}1_{\{2\delta_n \le |x_1 -x_2 |\le 5\delta_n /2\}}\ dx_1 \, dx_2 
\gtrsim \delta_n^{2-\beta}
\end{equation}
and
\begin{equation}\label{L1}
\int_{B_{\delta_n}}\int_{B_{\delta_n}}1_{\{ \sqrt{7\delta_n /2}\le |t_1 -t_2 |\le 2\sqrt{\delta_n} \}}\ dt_1 \, dt_2
 \gtrsim \delta_n^{2-2\gamma}\delta_n^{\gamma /2}.
\end{equation}
for a sequence $\delta_n$'s tending to $0$. (At the end of this proof we will say a few words 
about how to obtain $A$ and $B$.) Put $F=A\cup (A+1)$. Then
\begin{equation}\label{L2}
\int_{F_{\delta_n}}\int_{F_{\delta_n}}1_{\{2\delta_n \le 1-|x_1 -x_2 |\le 5\delta_n /2\}}\ dx_1 \, dx_2 
\gtrsim \delta_n^{2-\beta}.
\end{equation}
Let $K=F\times B$. Then \eqref{alphaset} holds with $\alpha ' =\beta +\gamma$ in place of $\alpha$
by our choices of $F$ and $B$.

Now 
$$
1-\delta \le \sqrt{(x_1 -x_2 )^2 +(t_1 -t_2 )^2}\le 1+\delta
$$
is equivalent to 
$$
\sqrt{(1-\delta )^2 -|x_1 -x_2 |^2 }\le |t_1 -t_2 |\le \sqrt{(1+\delta )^2 -|x_1 -x_2 |^2 }.
$$
If 
$$
2\delta\le 1-|x_1 -x_2 |\le 5\delta /2
$$
then 
$$
2\delta\le 1-|x_1 -x_2 |^2 \le 5\delta 
$$
and so if $\delta <1/2$ some algebra shows that  
\begin{equation*}
\sqrt{(1-\delta )^2 -|x_1 -x_2 |^2 }\le \sqrt{7\delta /2}<2\sqrt{\delta} \le \sqrt{(1+\delta )^2 -|x_1 -x_2 |^2 }.
\end{equation*}
Thus if 
$$
2\delta\le 1-|x_1 -x_2 |\le 5\delta /2  \text{ and }
 \sqrt{7\delta /2}\le |t_1 -t_2 |\le 2\sqrt{\delta}
 $$
it follows that 
$$
\sqrt{(1-\delta )^2 -|x_1 -x_2 |^2 }\le |t_1 -t_2 |\le \sqrt{(1+\delta )^2 -|x_1 -x_2 |^2 }.
$$
With \eqref{L2} and \eqref{L1} this gives 
$$
\int_{F_{\delta_n}}\int_{F_{\delta_n}}\int_{B_{\delta_n}} \int_{B_{\delta_n}}
1_{\{1-\delta_n \le \sqrt{(x_1 -x_2 )^2 +(t_1 -t_2 )^2}\le 1+\delta_n \}}\, dt_1 \, dt_2 \, dx_1 \, dx_2 
\gtrsim \delta_n^{4-(\beta +3\gamma /2)}
$$
and so  $|D^{\delta_n} |\gtrsim \delta_n^{4-(\beta +3\gamma /2)}$.

We conclude the proof of this lemma by describing a construction (which, though tedious, we include for the sake of completeness) of the required sets $F$ and $B$. For positive integers $p<q$ consider the Cantor set 
$C=C(p,q)$ constructed by removing $(2^p -1)$ equally spaced intervals open intervals  from $C_0 = [0,1]$
to obtain $C_1 =[0,2^{-q}]\cup\cdots\cup [1-2^{-q},1]$ and then continuing in the usual way, so that at the $j$th stage of the construction we have a set $C_j$ which is the union of $2^{jp}$ 
closed intervals of length $2^{-jq}$. Then 
\eqref{alphaset} holds with $C=\cap C_j$ instead of $K$ and with $\alpha =p/q$. Also, 
since $C_j \subset C+B(0,2^{-qj})=C_{2^{-qj}}$, for any $0<\kappa_1 <\kappa_2 <1$ we have 
\begin{equation*}
\int_{C_{2^{-qj}}}\int_{C_{{2^{-qj}}}}1_{\{\kappa_1 2^{-qj} \le |x_1 -x_2 |\le \kappa_2 2^{-qj} \}}\ dx_1 \, dx_2 
\gtrsim (2^{-qj})^{(2-p/q)}
\end{equation*}
and then also
\begin{equation}\label{L3}
\int_{C_{2^{-qj-2}}}\int_{C_{{2^{-qj-2}}}}1_{\{\kappa_1 2^{-qj} \le |x_1 -x_2 |\le \kappa_2 2^{-qj} \}}\ dx_1 \, dx_2 
\gtrsim (2^{-qj})^{(2-p/q)},
\end{equation}
where the implied constant depends on $\kappa_1$ and $\kappa_2$. One then sees that 
\begin{multline}\label{L4}
\int_{C_{2^{-2qj-2}}}\int_{C_{2^{-2qj-2}}}1_{\{\kappa2^{-qj} \le |x_1 -x_2 |\le 2^{-qj} \}}\ dx_1 \, dx_2 
\gtrsim \\
2^{2(p-q)j}(2^{-qj})^{(2-p/q)} =(2^{-2qj})^{(2-\frac{3}{2}\frac{p}{q})}.
\end{multline}

If $p_2$ and $q_2$ are chosen so that $\gamma =p_2 /q_2$, if $B=C(p_2 ,q_2 )$, and if
$$
\delta_n = 2^{-2q_1 q_2  n-2}
$$
then 
$$
\sqrt{\frac{7\delta_n} {2}}=\sqrt{\frac{7}{8}}\, 2^{-q_1 q_2 n},\ 2\sqrt{\delta_n}=2^{-q_1 q_2 n} 
$$
and so \eqref{L4} with $q=q_2$, $j=nq_1$, and
$\kappa =\sqrt{7/8}$ shows that \eqref{L1} holds. 

If $p_1$ and $q_1$ are chosen so that $\beta =p_1 /q_1$ and if $A=C(p_1 ,q_1 )$, then 
$$
2\delta_n =\frac{1}{2}2^{-2q_1 q_2 n},\, \frac{5}{2}\delta_n  =\small\frac{5}{8}2^{-2q_1 q_2 n}
$$
and so \eqref{L3} with $q=q_1$, $j=nq_2$, $\kappa_1 =1/2$, and $\kappa_2 = 5/8$ shows that \eqref{L0} holds. This completes the proof of the lemma.

\end{proof}

{\it Proof of Theorem \ref{d=3}:} The proof is similar to the proof of the bound $g_2 (\alpha )\leq 5\alpha /3$ of
Theorem \ref{d=2}. We begin by letting $K_m$, $\lambda =\lambda_n$, and the balls $B(c_n ,\delta )$, $1\le n\le N$ 
be the three-dimensional analogs of the quantities defined in the proof of Theorem \ref{d=2}. 
Instead of 
\eqref{UB2.1} we will now establish 
\begin{equation}\label{UB3.1}
\lambda \, N\, \delta^3 \lesssim \frac{\delta^{5(3-\alpha )}\delta^{\alpha /2}}{\lambda^3}.
\end{equation}
Interpolation with the trivial bound
\begin{equation}\label{trivest2}
\lambda \, N\, \delta^3 \lesssim \lambda\, \delta^{3-\alpha}
\end{equation}
gives 
\begin{equation}\label{UB3.2}
\lambda \, N\, \delta^3 \lesssim \lambda\, \delta^{6-15\alpha /8}.
\end{equation}
Then an argument completely analogous to the one in the proof of 
Theorem \ref{d=2} leads to $g_3 (\alpha )\le 15 \alpha /8$. 

Again choose a maximal $\delta$-separated subset $J$ of $K_\delta$. For each $c_n$ let 
\begin{equation*}
S_{c_n}=\{a\in J :1-3\delta \le |a-c_n |\le 1+3\delta \}=J\cap A(c_n ,3\delta )
\end{equation*}
Define 
\begin {multline*}
V=\\
\big\{(c_n ,a_1 ,a_2 ,a_3 ): 1\le n\le N,\, a_1 ,a_2 , a_3 \in S_{c_n},\, |a_i-a_j |\ge 
c\, \Big(\small\frac{\lambda}{\delta^{3-\alpha}}\Big)^{1/\alpha }
\text{ if }1\le i<j\le 3
\big\},
\end{multline*}
where $c$ is a small positive constant. We will prove \eqref{UB3.1} by again comparing upper and 
lower estimates for $|V|$. Before continuing we note that it suffices to prove \eqref{UB3.1} under the assumption that
\begin{equation}\label{largedist}
\Big(\small\frac{\lambda}{\delta^{3-\alpha}}\Big)^{1/\alpha }\gtrsim \delta^{1/2-\epsilon}
\end{equation}
for some small $\epsilon >0$ - otherwise \eqref{UB3.2} follows from \eqref{trivest2}. 
By using \eqref{alphaset} just as in the proof of \eqref{Tlower} we get the lower bound  
\begin{equation}\label{Tlower2}
|V|\gtrsim N\,\Big(\frac{\lambda}{\delta^3}\Big)^3 .
\end{equation}

As before we will obtain an upper bound for $|V|$ by controlling the multiplicity of the projection 
$$
(c_n ,a_1 ,a_2 ,a_3 )\mapsto (a_1 ,a_2 ,a_3 )
$$
of $V$ into $J^3$. In fact we will show that 
\begin{equation}\label{multiplicity2}
\text{$(c_n ,a_1 ,a_2 ,a_3 )\mapsto (a_1 ,a_2 ,a_3 )$ has multiplicity bounded by 
$C\,\delta^{-\alpha /2}\frac{\delta^{3-\alpha}}{\lambda}$.}
\end{equation}
Since $|J|\lesssim \delta^{-\alpha}$ it will then follow that
\begin{equation*}
|V|\lesssim \delta^{-3\alpha}\delta^{-\alpha /2}\frac{\delta^{3-\alpha}}{\lambda}.
\end{equation*}
Comparing this with \eqref{Tlower2} then gives \eqref{UB3.1}. Thus the proof of 
Theorem \ref{d=3} will be complete when \eqref{multiplicity2} is established.

We will establish \eqref{multiplicity2} by estimating the diameter of an intersection
\begin{equation}\label{shellintersect}
I\dot= A(a_1 ,3\delta )\cap A(a_2 ,3\delta )\cap A(a_3 ,3\delta ). 
\end{equation}
To begin, the intersection $A(a_1 ,0 )\cap A(a_2 ,0 )$ of the unit spheres centered at $a_1$ and $a_2$
is a circle contained in the hyperplane 
\begin{equation}\label{hplane}
P_{1,2}= \big\{x\in\bbR^3 :(x-a_1 )\cdot (a_2 -a_1 )=\frac{1}{2}|a_2 -a_1 |^2\big\}.
\end{equation}
If $x\in A(a_i ,3\delta )$ then $x\in A(a_i +e_i ,0)$ with $|e_i |\le 3\delta$. It follows that if 
$x\in A(a_1 ,3\delta )\cap A(a_2 ,3\delta )$, then 
\begin{equation*}
\big| (x-a_1 )\cdot (a_2 -a_1 )-\frac{1}{2}|a_2 -a_1 |^2 \big|\lesssim \delta ,
\end{equation*}
and so 
\begin{equation*}
A(a_1 ,3\delta )\cap A(a_2 ,3\delta )\subset P_{1,2}+B\big(0, C\frac{\delta} {|a_2 -a_1 |}\big).
\end{equation*}
Similarly, 
\begin{equation*}
A(a_1 ,3\delta )\cap A(a_3 ,3\delta )\subset P_{1,3}+B\big(0, C\frac{\delta} {|a_3 -a_1 |}\big).
\end{equation*}
If the $a_i$ are affinely independent, it follows that the intersection \eqref{shellintersect} is contained in an extrusion 
(in the direction perpendicular to $a_2 -a_1$ and $a_3 -a_1$) of a parallelogram $P$ contained in the plane 
$a_1 +\text{span} (a_2 -a_1 ,a_3 -a_1 )$. This parallelogram has two sides of length 
$\frac{C\delta}{|a_2 -a_1 |\sin (\theta )}$ perpendicular to $a_3 -a_1$ and 
two sides of length 
$\frac{C\delta}{|a_3 -a_1 |\sin (\theta )}$ perpendicular to $a_2 -a_1$ where $\theta\in (0,\pi )$ is the 
angle between $a_2 -a_1$ and $a_3 -a_1$.

We will need the estimate 
\begin{equation}\label{sine}
\sin (\theta )\gtrsim |a_3 -a_2|.
\end{equation}
This is the point at which \eqref{largedist} will come into play: we will be assuming that 
$
(c_n ,a_1 ,a_2 ,a_3 )\in V
$ 
and so it will follow that 
\begin{equation}\label{largeprod}
|a_2 -a_1 |, \, |a_3 -a_1 |, \, |a_3 -a_2 |\gtrsim \delta^{1/2-\epsilon}
\end{equation}
for some $\epsilon>0$. 
With no loss of generality we can write $a_1 =(0 ,0 ,0)$, $a_2 =(x_2 ,0 ,0)$, $a_3 =(x_3 ,y_3 ,0)$
and then assume that these points lie in the first octant, that $y_3 >0$,
and that $|a_2 |\ge |a_3 |$. We will now observe that if  $\sin (\theta )$ and therefore
$\tan (\theta )=y_3 /x_3$ are small compared to $|a_2 -a_3 |$, then the extrusion fails to intersect the 
shells $A(a_i ,3\delta )$.
To show this we begin by observing that  
the center $p$ of the parallelogram $P$ is the point of intersection of the perpendicular bisectors
of the segments $[a_1 ,a_2]$ and $[a_1 ,a_3 ]$ and has $y$ coordinate equal to 
$$
p_y \doteq\frac{y_3}{2}-\frac{x_3}{2y_3}(x_2 -x_3 )=\frac{y_3}{2}-\frac{1}{2\tan (\theta )}(x_2 -x_3 ).
$$
If $\tan (\theta )$ is small compared to $|a_2 - a_3|$, then $x_2 -x_3 \ge |a_2 -a_3 |/2$. Thus, 
it follows that $|p_y |$ is large. 
Since we have assumed that $|a_2 -a_1 | \geq |a_3 -a_1|$, the diameter of $P$ is bounded by 
$$
\frac{2C\delta}{|a_3 -a_1 |\sin (\theta )}\lesssim \frac{\delta^{1/2+\epsilon}}{\sin (\theta )},
$$
where we have used \eqref{largeprod}.
This will be small compared to $|p_y |$ (since
$$
\frac{|x_2 -x_3 |}{\tan (\theta )}\gtrsim \frac{\delta^{1/2-\epsilon}}{\sin (\theta )},
$$
again by \eqref{largeprod}). In this case the distance $\rho$ from $P$ to the $x$-axis will be comparable to $|p_y |$. But if $\rho>2$, say, the extrusion will miss the shells $A(a_i ,3\delta )$
(whose centers lie in the $xy$-plane above the $x$-axis).

With \eqref{sine} it now follows from the definition of $V$ that the diameter of $P$ is bounded by 
$C\delta\big(\frac{\delta^{3-\alpha}}{\lambda}\big)^{2/\alpha}$. The following estimate is a consequence of
the subadditivity of the function $\sqrt \cdot$ on $(0,\infty)$:
\begin{multline*}
\Big|\sqrt {(1+\epsilon_1 )^2 -(x_1 ^2 +y_1^2 )}-\sqrt {(1+\epsilon_2 )^2 -(x_2 ^2 +y_2^2 )}\Big|\le \\
\sqrt{\big|2\epsilon_1 +\epsilon_1^2 -2\epsilon_2 -\epsilon_2^2 \big|} +\sqrt {\big|
x_2^2 +y_2^2 -(x_1^2 +y_2^2 )\big|}.
\end{multline*}
With $|\epsilon_1 |,|\epsilon_2 |\le 2\delta$ this shows that if $(x_i ,y_i ,z_i )\in A(0,\delta )$ for $i=1,2$ and
$|(x_1 ,y_1 )-(x_2 ,y_2 )|\le \kappa$ then 
$$
|(x_1 ,y_1 ,z_1 )-(x_2 ,y_2 ,z_2)|\lesssim\max (\delta^{1/2},\kappa^{1/2}).
$$
Thus it follows from our bound on the diameter of $P$ that 
\begin{equation}\label{diamest}
\text{diam} (I)\le C\, \Big(\delta\big(\frac{\delta^{3-\alpha}}{\lambda}\big)^{2/\alpha}\Big)^{1/2}
=C\,\delta\,\delta^{-1/2}\big(\frac{\delta^{3-\alpha}}{\lambda}\big)^{1/\alpha}. 
\end{equation}
Now if $(c_n ,a_1 ,a_2 ,a_3 ), (c_{n'},a_1 ,a_2 ,a_3 )\in T$, we have $c_n ,c_{n'}\in I$. Thus
\eqref{diamest}, the fact that the $c_n$'s are $\delta$-separated, and \eqref{alphaset} 
together yield \eqref{multiplicity2}. This completes the proof of Theorem \ref{d=3}.

\end{document}